\documentclass[11pt,a4paper]{amsart}
\usepackage[margin=2cm]{geometry}

\usepackage{amsmath, amsthm, amssymb}
\input xy
\xyoption{all}
\usepackage{mathrsfs}
\usepackage{enumerate}
\usepackage{caption}
\usepackage{comment}
\usepackage{subcaption}
\usepackage{graphicx}
\usepackage{color}
\usepackage{tikz}
\usepackage{tikz-cd}
\usepackage{tkz-euclide}

\usepackage{scalerel}

\usepackage[driverfallback=hypertex]{hyperref}
\usepackage{nameref,zref-xr}                    

\newtheorem{thm}{Theorem}[section]
\newtheorem{prop}[thm]{Proposition}

\newtheorem{lem}[thm]{Lemma}

\newtheorem{lemma}[thm]{Lemma}
\newtheorem{cor}[thm]{Corollary}

\theoremstyle{definition}
\newtheorem{definition}[thm]{Definition}
\newtheorem{nn}[thm]{Notation}

\theoremstyle{remark}
\newtheorem{rmk}[thm]{Remark}

\newtheorem{ex}[thm]{Example}

\newcommand{\op}[1]{\operatorname{#1}}
\newcommand{\Aut}{\op{Aut}}

\renewcommand{\Aut}{\text{Aut}}

\newcommand{\ttts}{\mathbf{t}}

\newcommand{\N}{\mathbb{Z}_{\ge 0}}

\newcommand{\Z}{\ensuremath{\mathbb{Z}}}

\newcommand{\C}{\ensuremath{\mathbb{C}}}
\newcommand{\R}{\ensuremath{\mathbb{R}}}

\renewcommand{\d}{\ensuremath{\mu}}

\renewcommand{\i}{\sqrt{-1}}

\usepackage{geometry}
\numberwithin{equation}{section}

\begin{document}

\title{Mirror Symmetry for open $r$-spin invariants}

\author{Mark Gross}
\address{M. Gross: \newline Department of Pure Mathematics and Mathematical Statistics, University of Cambridge, CB4 0WB, United Kingdom}
\email{mgross@dpmms.cam.ac.uk}

\author{Tyler L. Kelly}
\address{T. L. Kelly: \newline School of Mathematics, University of Birmingham, Edgbaston, B15 2TT, United Kingdom}
\email{t.kelly.1@bham.ac.uk}

\author{Ran J. Tessler}
\address{R. J. Tessler:\newline Department of Mathematics, Weizmann Institute of Science, Rehovot 76100, Israel}
\email{ran.tessler@weizmann.ac.il}

\begin{abstract}
We show that a generating function for open $r$-spin enumerative invariants
produces a universal unfolding of the polynomial $x^r$.
Further, the coordinates parametrizing this universal unfolding are
flat coordinates on the Frobenius manifold associated to the Landau-Ginzburg
model $(\C,x^r)$ via Saito-Givental theory.
This result provides evidence for the same phenomenon to occur in higher dimension, proven in the sequel \cite{GKT}.
\end{abstract}

\maketitle

\setcounter{tocdepth}{1}

\section{Introduction}

Buryak, Clader, and Tessler recently constructed an open $r$-spin enumerative theory \cite{BCT:I, BCT:II}, following the case of descendent invariants
on the moduli space of holomorphic disks developed by Pandharipande, Solomon and Tessler in \cite{PST14}. Roughly stated, they construct a moduli space of $r$-stable orbidisks with $r$-spin structures that have prescribed twists at both internal and boundary marked points. In turn, this gives a closed form expression for the open $r$-spin invariants
$$
\left\langle \prod_{i=1}^l \tau_0^{a_i} \sigma^{k+1}\right\rangle^{\tfrac1r, o}
$$
 corresponding to genus $0$ orbidisks with $k+1$ boundary marked points with twist $r-2$ and $l$ internal marked points with twists $a_1, \ldots, a_l$.

These invariants are analogous to the closed $A$-model enumerative theory for Landau-Ginzburg models constructed  in a sequence of papers \cite{JKV, FJR3, FJR2, FJR}, using ideas of Witten \cite{Wit93}. In these papers, the authors build an enumerative theory in the case of gauged Landau-Ginzburg models $(\C^n, W, G)$ where $W$ is an invertible polynomial and $G$ is a subgroup of the diagonal automorphism group of $W$.
 Berglund and H\"ubsch predicted mirror pairs between such Landau-Ginzburg models. They propose that the mirror to a gauged LG model $(\C^n, W, G)$ should be a mirror gauged LG model $(\C^n, W^T, G^T )$, where $W^T$ is the so-called transposed polynomial and $G^T$ is the dual group \cite{BH92}.
One such example is the mirror pair
\begin{equation}
\label{eq:mirror pair}
(\C^n, x_1^{r_1}+\cdots+x_n^{r_n},\mu_{r_1}\times\cdots\times\mu_{r_n})
\leftrightarrow (\C^n, x_1^{r_1}+\cdots+x_n^{r_n}, 1).
\end{equation}
The open $r$-spin invariants stated above correspond to the open $A$-model invariants for the Landau-Ginzburg model $(\C, x^r, \mu_r)$, following the analogue of the closed case established in \cite{FJRrspin}.
 Thus from ~\eqref{eq:mirror pair} when $n=1$, the open $r$-spin invariants should correspond to an open $B$-model enumerative theory $(\C^1, x^r, 1)$. 

The (closed) $B$-model side of the story for the right-hand side of
the correspondence \eqref{eq:mirror pair} was developed in \cite{Sai81, Sai83, Giv, Dubrovin}, and was more recently used in Landau-Ginzburg mirror symmetry in \cite{LLSS, HeLiShenWebb}. The $B$-model side, put as simply as possible, is
a Saito-Givental
theory, which involves calculating
oscillatory integrals of the form
\begin{equation}
\label{eq:osc integral}
\int_{\Gamma} e^{W_{\bf t}/\hbar} f(x_1,\ldots,x_n,{\bf t}) dx_1\wedge
\cdots \wedge dx_n.
\end{equation}
Here $x_1,\ldots,x_n$ are coordinates on $\C^n$, $\hbar$ is a coordinate
on an auxiliary $\C^*$, $f:\C^n\rightarrow\C$
is a carefully chosen regular function, and ${\bf t}=\{t_{ij}\,|\,1\le i\le n, 0\le j \le r_i-2\}$ is
a set of coordinates on the parameter space for a universal unfolding
$W_{\bf t}$ of $W$.
Finally $\Gamma$ runs over some suitable non-compact cyles in $\C^n$.
The requirement on $f$ is that the form $f dx_1\wedge\cdots \wedge dx_n$
is a so-called primitive form in the sense of Saito-Givental theory.

While Saito-Givental theory in general gives a Frobenius manifold
structure to the universal unfolding of $W$, determining this structure
can be quite difficult. However, experience with mirror
symmetry for toric Fano varieties \cite{GrossP2,FOOO1} suggests that
mirror symmetry becomes much more transparent when a specific perturbation
$W_{\bf t}$ of the original potential is used. Previous work on the Fano/LG mirror correspondence (see e.g., \cite{ChoOh}, \cite{FOOO1}, \cite{GrossP2})
found that there is a ``correct'' universal unfolding $W_\mathbf{t}$
which is a generating function for counting Maslov index two disks
with boundary on a Lagrangian torus. The advantage of this ``correct''
universal unfolding is that there is a canonical choice of primitive form
(which in our case will just be $\Omega= dx$) and that flat coordinates
coincide with a natural choice of coordinates. Using the same philosophy here, we show that the ``correct'' universal unfolding is a generating function for the open $r$-spin invariants.

\begin{thm}\label{thm:A}
The deformed potential
\begin{equation}\label{flat one dim}
W_{\mathbf{t}} = \sum_{k \geq 0, l \geq 0} \sum_{\{a_i\} \in \mathscr{A}_l} (-1)^{l-1} \frac{\left\langle \prod_{i=1}^l \tau_0^{a_i} \sigma^{k+1}\right\rangle^{\tfrac1r, o}}{k!|\mathrm{Aut}(\{a_i\})|} \left(\prod_{i=1}^l t_{a_i} \right) x^{k},
\end{equation}
has the primitive form $\Omega=\textup{d}x$ and $t_0,\ldots,t_{r-2}$ are flat coordinates for the Frobenius manifold constructed via Saito-Givental theory for the LG model $W=x^r$.
\end{thm}

In this paper, we use the closed form for open $r$-spin invariants in \cite{BCT:II} as a black box to prove the mirror Theorem~\ref{thm:A}. The flat structure of the Frobenius manifold for the Landau-Ginzburg model $(\C^1, x^r, 1)$ has been studied in the past in its own right in integrable systems. In the 1980s, Noumi and Yamada found algorithms for finding the flat coordinates \cite{Noumi, NoumiYamada}. These techniques were transcendental in nature, whereas we prove Theorem~\ref{thm:A} through purely algebraic and combinatorial means once given the open $r$-spin invariants.  We aim for this paper to serve as a link from the integrable systems literature for flat structures to the Landau-Ginzburg mirror symmetry analogue of the Fano/LG correspondence developed in the $n=2$ case in our sequel paper \cite{GKT}.

In \cite{Bur20}, Buryak describes the flat structure Frobenius manifold via the extended $r$-spin invariants introduced in \cite{BCT_Closed_Extended}. There, he shows that the change of coordinates from the versal deformation to the flat coordinates can be derived from particular differentials of a generating function built from extended invariants (Theorem 3.1 in \cite{Bur20}).  The extended invariants are closely related to the open $r$-spin invariants \cite[Theorem 1.3]{BCT:II}. The proof presented here is different from that in \cite{Bur20} as the author uses transcendental techniques from integrable systems.
 While it is true that, in dimension one, one can derive the flat structure using either extended or open $r$-spin invariants, the generating function using open invariants is more direct and one is no longer able to use extended invariants when one considers singularities that are not ADE. In the $n=2$ case constructed in \cite{GKT}, we use open FJRW invariants to construct flat coordinates for the singularity $W = x_1^{r_1} + x_2^{r_2}$. These in turn can be used to compute the closed extended invariants (with descendents). We aim for this paper to help solidify the import of the perspective taken from the Fano/LG mirror correspondence to Landau-Ginzburg mirror symmetry. 

 In Section 2, we outline the relevant Saito-Givental theory for the Landau-Ginzburg models $(\C^n, x_1^{r_1}+\cdots+x_n^{r_n}, 1)$. In Section 3, we provide the definition of primitive forms and flat coordinates in the context of the LG models in~\eqref{eq:mirror pair} and give an example. In Section 4, we prove Theorem~\ref{thm:A}.

\medskip

\noindent {\bf Acknowledgments.} The authors would like to thank Alexander Buryak and Robert Maher for discussions relating to this work. The first author acknowledges support from the
EPSRC under Grants EP/N03189X/1, a Royal Society Wolfson Research Merit Award,
and the ERC Advanced Grant MSAG. The second author acknowledges that this paper is based upon work supported by the UKRI and EPSRC under fellowships MR/T01783X/1 and EP/N004922/2.
The third author, incumbent of the Lillian and George Lyttle Career Development Chair, acknowledges support provided by the ISF grant No. 335/19 and by a research grant from the Center for New Scientists of Weizmann Institute.

\section{The $B$-Model State Space for Fermat polynomials}
\label{sec:B model}
In this section, we describe the enumerative theory associated to a Landau-Ginzburg $B$-model, due to Saito and Givental and
described in the case of the $B$-model of FJRW theory by He, Li, Li, Saito, Shen
and Webb in \cite{LLSS, HeLiShenWebb}. \label{subsec:state space}
Let $(X, W)$ be a Landau-Ginzburg model, i.e., $X$ a variety and
$W:X \rightarrow \C$ a regular function.
We will not allow for a group of symmetries. We shall quickly review
Saito-Givental theory in this context. For a much more in-depth
exposition in our framework and notation, see Chapter 2 of \cite{kansas} and references therein.
 A principal object of study is the twisted de Rham complex
\[
(\Omega_X^\bullet, d + \hbar^{-1}dW\wedge - ),
\]
where $\Omega_X^i$ is the sheaf of algebraic
$i$-forms on $X$ and $\hbar\in \C^*$ is an auxiliary parameter. We restrict our attention to the Landau-Ginzburg model
\begin{equation}\label{r-spin LG B-model}
(X, W) = (\C^n, W=\sum_ix_i^{r_i}).
\end{equation}

\begin{prop}\label{hypercoh for fermat}
Consider the Landau-Ginzburg model $(X,W)$ in \eqref{r-spin LG B-model}. Then the hypercohomology group $\mathbb{H}^n(X, (\Omega_X^\bullet, d + \hbar^{-1}dW\wedge -))$ has dimension $\prod_{i=1}^n (r_i-1)$ and is generated by the basis
$$
M = \left\{\prod_i x_i^{a_i} \Omega \ \middle| \ 0 \leq a_i \leq r_i - 2\right\},
$$
where $\Omega = dx_1 \wedge \cdots \wedge dx_n$.
\end{prop}

\begin{proof}
First, since $\C^n$ is affine, the cohomology of the sheaves
$\Omega_{\C^n}^i$ vanishes in degree at least one. Thus by the hypercohomology
spectral sequence, it is enough to compute the cohomology of the complex
$$
0 \rightarrow  \Omega_{\C^n}^0 \overset{\delta
}{\longrightarrow } \cdots \overset{\delta
}{\longrightarrow } \Omega_{\C^n}^{n-1} \overset{\delta}{\longrightarrow } \Omega_{\C^n}^n \rightarrow 0,
$$
where $\delta = d + \hbar^{-1}dW\wedge -$. Thus the hypercohomology we are interested in can be written as
\begin{equation}\label{the hypercoh}
\mathbb{H}^n(X, (\Omega_X^\bullet, d + \hbar^{-1}dW\wedge -)) = \Omega_{\C^n}^n / \delta(\Omega_{\C^n}^{n-1}).
\end{equation}
First note that
\begin{equation}
\delta(x_1^{a_1} \cdots x_{i-1}^{a_{i-1}} x_{i+1}^{a_{i+1}} \cdots x_n^{a_n} dx_1 \wedge \cdots \wedge \widehat{dx_i} \wedge \cdots \wedge dx_n) = (-1)^{i-1} \hbar^{-1} r_i x_1^{a_1} \cdots x_{i-1}^{a_{i-1}} x_i^{r_i-1} x_{i+1}^{a_{i+1}} \cdots x_n^{a_n} \Omega.
\end{equation}
Next, when $a_1\ldots a_n \in \N$ and $a_i >0$ we have that
\begin{equation}\begin{aligned}
\delta( x_1^{a_1} \cdots x_n^{a_n} dx_1 \wedge \cdots \wedge \widehat{dx_i} \wedge \cdots \wedge dx_n) = & {}  (-1)^{i-1} a_i x_1^{a_i} \cdots x_i^{a_i-1} \cdots x_n^{a_n} \Omega \\
	& {} + (-1)^{i-1} \hbar^{-1} r_i x_1^{a_1} \cdots x_i^{r_i + a_i - 1} \cdots x_n^{a_n} \Omega
\end{aligned}\end{equation}
Thus we have in the quotient $ \Omega_{\C^n}^n / \delta(\Omega_{\C^n}^{n-1})$ the relations
\begin{equation}\begin{aligned}\label{twisted deRham IBP}
x_1^{a_1} \cdots x_{i-1}^{a_{i-1}} x_i^{r_i-1} x_{i+1}^{a_{i+1}} \cdots x_n^{a_n} \Omega &= 0 \\
x_i^r x_1^{a_1} \cdots x_n^{a_n} \Omega &= - \hbar \tfrac{a_i+1}{r_i} x_1^{a_1} \cdots x_n^{a_n} \Omega.
\end{aligned}\end{equation}
Thus, any section of $\Omega_{\C^n}^n$ of the form
$$
x_1^{a_1} \cdots x_n^{a_n} \Omega
$$
with $a_i \ge r_i$ can have its exponent reduced by $r_i$ and any section with $a_i = r_i-1$ vanishes, yielding the result.
\end{proof}

\begin{rmk}
Note that the right-hand side of \eqref{the hypercoh} is the exact expression for the formally completed version of the Brieskorn lattice $\mathcal{H}_W^{(0)}$ associated to $W$ stated in \textsection3.1 of \cite{LLSS}.
\end{rmk}

Using Proposition~\ref{hypercoh for fermat}, there is a vector bundle ${\mathcal R}^{\vee}$ on the
$\hbar$-plane $\C$ whose fibre over $\hbar\in \C^*$ is
the cohomology group $\mathbb{H}^n(X, (\Omega_X^\bullet, d+ \hbar^{-1}dW\wedge-))$, and $M$ yields a frame for $\mathcal{R}^{\vee}$
which extends across the origin; see \cite{kansas}, \S2.2.2 and in particular
Definition 2.37.

There is a homology group dual to the hypercohomology group  $\mathbb{H}^n(X, (\Omega_X^\bullet,  d+\hbar^{-1}dW\wedge-))$ given by the relative homology
$$
H_n(X, \op{Re} W/ \hbar \ll 0; \C).
$$
Indeed, there is a natural perfect pairing
\begin{equation}\begin{aligned}\label{perfectpairing}
H_n(X, \op{Re} W/ \hbar \ll 0; \C) \times \mathbb{H}^n(X, (\Omega_X^\bullet,
d+\hbar^{-1} dW\wedge-)) &\longrightarrow \C, \\
(\Xi, \omega) &\longmapsto \int_{\Xi} e^{W / \hbar} \omega.
\end{aligned}\end{equation}
Thus there must be a dual basis for $H_n(X, \op{Re} W/ \hbar \ll 0; \C)$ for any basis of the hypercohomology group.

\begin{ex}\label{exampleDualBasis}
Consider the Landau-Ginzburg model $(\C, x^r)$. Given a fixed value of $\hbar \in \C^*$, set $\Psi_j := \{ te^{((\pi + \arg \hbar)\i + 2\pi j \i)/r} \ | \ t \in \R_{\geq 0}\}$.
Note this requires a choice of $\arg \hbar$: if $\arg \hbar$ is replaced
with $\arg\hbar + 2\pi$, $\Psi_j$ becomes $\Psi_{j+1}$. This will
yield a multi-valuedness for the cycles $\Xi_j^r$ constructed below.
In the formulas below, we make a specific choice of branch of $\arg\hbar$,
which will then give a well-defined choice of $\hbar^{1/r}$.

We now show there exists cycles $\Xi^r_j \in H_1(X, \op{Re} (x^r/ \hbar) \ll 0; \C)$ for $0\leq i \leq r-2$ so that
\begin{equation}\label{eqn: dual basis}
\int_{\Xi^r_j} x^k e^{x^r/ \hbar} dx = \delta_{jk}
\end{equation}
for all $0\leq j \leq r-2$, where $\delta_{jk}$ is the Kronecker delta function.

First note that $W/\hbar$ is real and negative on $\Psi_j$, going to
$-\infty$ in the unbounded direction of $\Psi_j$. Thus
$\Psi_{j+1} - \Psi_j \in H_1(X, \op{Re} W/ \hbar \ll 0; \C)$. By a direct
computation of exponential integrals, we find that
$$
A_{jk}:=\int_{\Psi_{j+1} - \Psi_j} x^k e^{x^r/ \hbar} dx =\frac{1}{r}C_k \zeta^{ j(k+1)},$$
where $\zeta:=e^{2\pi \i / r}$ and
$$
C_k = e^{ \pi i\frac{k+1}{r}}\hbar^{(k+1)/r} \Gamma\left(\frac{k+1}{r}\right) \left(e^{2\pi i(k+1) / r} - 1\right).
$$

 We find the inverse of the matrix $A=(A_{jk}) \in \op{Mat}_{(r-1) \times (r-1)}(\C(\hbar))$ to explicitly compute the cycles $\Xi_j$ as a linear combination of the $(\Psi_{j+1} - \Psi_j)$. Let $C = \op{diag}(C_k)_{k=0}^{r-2}$. If we postmultiply by $C^{-1}$, we obtain the matrix
$$
B:= AC^{-1} = \left(\frac{1}{r}\zeta^{ j(k+1)}\right)_{0\le j,k\le r-2}.
$$
The inverse matrix here can be computed to be $(B^{-1})_{jk} =(\zeta^{-k(j+1)} - \zeta^{j+1}).$ So then
$$
(A^{-1})_{jk} = (C^{-1}B^{-1})_{jk} = \frac{1}{C_j}(\zeta^{-k(j+1)} - \zeta^{j+1}).
$$
We then have an explicit description of $\Xi_j^r$:

\begin{equation}\label{ExplicitBasis}
\Xi_j^r := \frac{1}{e^{\pi i \tfrac{j+1}{r}} \hbar^{(j+1)/r}\Gamma(\tfrac{j+1}{r})(\zeta^{j+1} - 1)}\sum_{k = 0}^{r-2}(\zeta^{-k(j+1)} - \zeta^{j+1}) (\Psi_{k+1} - \Psi_k)
\end{equation}
\end{ex}

The following more general integrals will also be important for us in proving Theorem~\ref{thm:A}:

\begin{lemma}
\label{lem:int by parts}
For all $r \in \mathbb{N}$ and $n \in \Z$:
\begin{align}\label{r spin integration by parts}
\begin{split}
\int_{\Xi_d}  x^{nr+k} e^{x^r/\hbar}dx = {} &
(-1)^n\hbar^n \left(\prod_{i=1}^n (i -1+ \frac{k+1}{r})\right) \int_{\Xi_d}  x^{k} e^{x^r/\hbar}dx\\
= {} & (-1)^n\hbar^n \frac{\Gamma\left(n+\frac{k+1}{r}\right)}{
\Gamma\left(\frac{k+1}{r}\right)}\int_{\Xi_d}  x^{k} e^{x^r/\hbar}dx.
\end{split}
\end{align}
\end{lemma}

\begin{proof}
This is proven by iterating integration by parts or the relation given in~\eqref{twisted deRham IBP}.
\end{proof}

Using this example, we can describe a similar set of cycles
for the Fermat Landau-Ginzburg model
$$
(\C^n, W_0), \qquad W_0 :=x_1^{r_1} + \cdots + x_n^{r_n} .
$$
Define the set $D := \{\d = (\d_1, \ldots, \d_n) \in \N^n \ | \ 0\leq \d_i \leq r_i - 2\}$.
For any $\d \in D$, we then can define the cycles
\begin{equation}\label{Xids}
\Xi_{\d} := \Xi_{\d_1}^{r_1} \times \cdots \times \Xi_{\d_n}^{r_n}
\end{equation}
 as products of those defined in the previous example. Then we have that
\begin{equation}\label{dual bases oscillatory integrals}
\int_{\Xi_{\d}} x^{\d'} e^{(x_1^{r_1} + \cdots + x_n^{r_n}) / \hbar} \Omega = \delta_{\d \d'}
\end{equation}
with $\Omega=dx_1\wedge\cdots\wedge dx_n$.

When $n=1$ and $W = x^r$, we will suppress notation and take $\Xi_j := \Xi^r_j$.


\section{Flat coordinates for Fermat polynomials} We restrict to the case where $(X, W_0) = (\C^n, x_1^{r_1} + \cdots + x_n^{r_n})$. Note that the elements $x^{\d}$ for $\d \in D$ are a basis for the Jacobian ring
$\C[x_1,\ldots,x_n]/(\partial W_0/\partial x_1,\ldots,\partial W_0/\partial x_n)$. Thus we may write a universal unfolding of $W_0$ parameterized by
a germ $\mathcal{M}$ of the origin in the Jacobian ring, viewed as a
vector space.
We use coordinates $y_\d$ on $\mathcal{M}$, parameterized by $\d \in D$.
Here $\{y_{\d}\}$ is the dual basis to
$\{x^{\d}\}$.
The versal deformation $W$ for $W_0$ on $\mathcal{M} \times X$ is then given by
$$
W =  x_1^{r_1} + \cdots + x_n^{r_n} + \sum_{\d \in D} y_{\d} x^{\d}.
$$

For a polynomial function $W:\C^n\rightarrow\C$ ,
we now consider the following collection of oscillatory integrals, each of
which we can view as a formal power series in the variables $y_\d$:
$$
\int_{\Xi_{\d}}  e^{W / \hbar}f \Omega  = \sum_{j=-\infty}^\infty \varphi_{\d, j}(\mathbf{y}) \hbar^{-j}
$$
where $\varphi_{\d,j}(\mathbf{y}) \in \C [[y_\d]]$.

The following is an oversimplification of Saito's theory of primitive
forms, but is sufficient for our purposes:

\begin{definition}
If $\varphi_{\d,j}\equiv 0$ for all $j<0$ and
$\varphi_{\d,0}=\delta_{\d\mathbf{0}}$, then we say
that $f\Omega$ is a \emph{primitive form}. Further, in this case,
$\varphi_{\d,1}$ form a set of coordinates on the universal unfolding
$\mathcal{M}$ called \emph{flat coordinates}.

\begin{nn}
We typically will use the variables $t_\d := \varphi_{\d, 1}$ for the flat coordinates.
\end{nn}
\end{definition}

\begin{ex}\label{exm: quartic}
Consider the Landau-Ginzburg model $(\C, x^4)$. In this case, we have
$M = \{ 1, x, x^2\}$ and we consider the versal deformation
$$
W = x^4 + y_2x^2 + y_1x + y_0.
$$
Using Lemma \ref{lem:int by parts},
we obtain the following expansions showing the lowest degree terms
in $\hbar^{-1}$:
\begin{equation}
\begin{aligned}
\int_{\Xi_{0}}  e^{W / \hbar} \Omega  &= 1 + (y_0-\tfrac18 y_2^2) \hbar^{-1} + O(\hbar^{-2}) \\
\int_{\Xi_{1}}  e^{W / \hbar} \Omega  &= y_1 \hbar^{-1} + O(\hbar^{-2}) \\
\int_{\Xi_{2}}  e^{W / \hbar} \Omega  &= y_2 \hbar^{-1} + O(\hbar^{-2})
\end{aligned}
\end{equation}
\end{ex}
Thus in this case $\Omega$ is already a primitive form, and flat coordinates
are given by $t_0=y_0-y_2^2/8$, $t_1=y_1$, $t_2=y_2$. However,
we may then rewrite the universal unfolding using this change of variables,
obtaining
$$
W_\mathbf{t} = x^4 + t_2 x^2 + t_1 x + t_0 + \tfrac{1}{8}t_2^2.
$$
Then
$$
\int_{\Xi_{d}}  e^{W_\mathbf{t} / \hbar} \Omega  = \delta_{0d} + t_d \hbar^{-1} + \dots
$$
for all $d \in\{0,1,2\}$.

\begin{rmk}
He, Li, Shen and Webb in
\cite{HeLiShenWebb} use
Saito's general framework for constructing Frobenius manifold structures
on the universal unfoldings of potentials
(see \cite{Sai81, Sai83} or Section III.8 of \cite{Manin}) to construct
the $B$-model  Frobenius manifold for LG models $(\C^n, W)$, where $W$ is an invertible polynomial. Our definition of flat
coordinates coincides with the definition given in
Equation (13) in \S2.2.2 of \cite{HeLiShenWebb}, with
different notation.
\end{rmk}

\section{Open Landau-Ginzburg Mirror Symmetry in dimension 1}

In the case of the mirror pair of Landau-Ginzburg models
$$
(\C, x^r, \mu_r) \longleftrightarrow (\C, x^r),
$$
 the relevant open invariants
for the Landau-Ginzburg model $(\C, x^r, \mu_r)$,
i.e.,
open $r$-spin invariants, have been already constructed in \cite{BCT:I, BCT:II}. Recall the following theorem:

\begin{thm}[Theorem 1.2 of \cite{BCT:II}]\label{BCT Computation}
Take $k, l \geq 0$ and $0 \leq a_1,\ldots, a_l \leq r-1$. Suppose we consider the open $r$-spin invariant associated to $k+1$ boundary marked points with twist $r-2$ and $l$ internal marked points with twists $a_1, \ldots a_l$. Then we have
$$
\left\langle \prod_{i=1}^l \tau_0^{a_i} \sigma^{k+1}\right\rangle^{\tfrac1r, o} = \begin{cases} \frac{(k+l-1)!}{(-r)^{l-1}},  &\text{if $k\geq 0$ and $\frac{(r-2)k + 2\sum_i a_i}{r} = 2l+k-2$}; \\ 0, &\text{otherwise}.\end{cases}
$$
\end{thm}

\begin{prop}
\label{rem:finiteness}
If the open $r$-spin invariant $\left\langle \prod_{i=1}^l \tau_0^{a_i} \sigma^{k+1}\right\rangle^{\tfrac1r, o}$ is nonzero, then $k \le r$ and $l \le r$. In particular, the number of  nonzero primary open $r$-spin invariants is finite.

\end{prop}
\begin{proof} The constraint in Theorem \ref{BCT Computation}
\begin{equation}\label{r spin constraint on boundary markings}
\frac{(r-2)k + 2\sum_i a_i}{r} = 2l+k-2
\end{equation}
requires that $2r = 2(l r-\sum_i a_i) + 2k$.
Since $lr -\sum_i a_i \ge l$, there are at most $r+1$ boundary marked points
and at most $r$ internal marked points. \end{proof}

We also have the following lemma:

\begin{lem}\label{twist Conditions}
Suppose $I = \{a_1,\ldots, a_l\}$ is a nonempty multiset of internal marking twists with $0 \le a_i \le r -2$. Let $r(I)\in\{0,\ldots, r-1\}$  be such that $\sum_i a_i \equiv r(I) \pmod r$. Then the open $r$-spin invariant $\left\langle \prod_{i=1}^l \tau_0^{a_i} \sigma^{k+1}\right\rangle^{\tfrac1r, o}$ is nonzero if and only if $k = r(I)$ and $\sum_i a_i = r(I) + (l-1)r$.
\end{lem}

\begin{proof}
By Observation 2.1 of \cite{BCT:I}, in order for the moduli space corresponding to the open $r$-spin invariant $\left\langle \prod_{i=1}^l \tau_0^{a_i} \sigma^{k+1}\right\rangle^{\tfrac1r, o}$ to be nonempty,
the following integrality condition must hold:
$$
e:= \frac{2 \sum_i a_i + k (r-2)}{r} =  \frac{2( \sum_i a_i -k)}{r}+ k  \in \Z.
$$
Moreover, by Observation 2.10 of \cite{BCT:I}, we must also have that $e \equiv k \pmod 2$.  This implies that $\frac{2(\sum_i a_i -k)}{r}$ is an even integer, hence $\sum_i a_i - k \in r\Z$, so $k\equiv r(I) \pmod r$.

 By Proposition
\ref{rem:finiteness}, we have $k\le r$, so
$k=r(I)$ unless $k=r$. If $k=r$, then, by ~\eqref{r spin constraint on boundary markings},
we must have $\sum_i a_i = rl$ hence $l=0$. However, we are assuming the set of internal
markings is non-empty, so this does not occur. Thus $k= r(I)$.


Now, suppose that $\sum_i a_i = r(I) + pr$ for some $p$. Then we have
that
\begin{equation}
\frac{(r-2)k + 2(r(I) + pr)}{r} = \frac{(r-2)r(I) + 2r(I) + 2pr}{r} = r(I) + 2p.
\end{equation}
Note by Theorem \ref{BCT Computation}
we must have $r(I) + 2p=2l + k-2$ in order to have a nonzero open $r$-spin invariant, but since $k = r(I)$ this implies that $2p = 2l -2$, or $p = l-1$
as desired.

The converse is a straightforward computation from Theorem \ref{BCT Computation}.
\end{proof}
\begin{cor}\label{Combo Corollary}
Consider a multiset $I$ as in Lemma~\ref{twist Conditions} so that $\sum_i a_i = r(I) + (|I|-1)r$.
\begin{enumerate}[(a)]
\item For any $I' \subset I$, we have $\sum_{i \in I'} a_i = r(I') + (|I'| -1) r$. Moreover, there exists a unique $k_{I'}\in \{0,\ldots, r-1\}$ so that $\left\langle \prod_{i\in I'}  \tau_0^{a_i} \sigma^{k_{I'}+1}\right\rangle^{\tfrac1r, o}$ is nonzero.
\item For a multiset partition $I = I_1\cup \cdots \cup I_h$, we have
$r(I_1) + \cdots + r(I_h) = r(I) + (h-1)r$.
\end{enumerate}
\end{cor}

\begin{proof}
We first prove item (a). Note the condition
that $\sum_{i\in I'}a_i =r(I')+(|I'|-1)r$ for any $I'\subseteq I$
is equivalent to $\sum_{i\in I'}(r-a_i) = r-r(I')$. As
$r(I')\in \{0,\ldots,r-1\}$, this is equivalent to
$0<\sum_{i\in I'} (r-a_i) \le r$. The left-hand inequality is automatic
as $a_i<r$ for all $i$. We have $\sum_{i\in I}(r-a_i)\le r$ by
assumption, and hence $\sum_{i\in I'}(r-a_i)\le r$ for any subset
$I'\subseteq I$. Thus $\sum_{i\in I'}a_i=r(I')+(|I'|-1)r$.


For (b), note that $r(I)+(|I|-1)r=\sum_{i\in I} a_i
=\sum_j \sum_{i\in I_j} a_i=\sum_j (r(I_j)+(|I_j|-1)r) = (|I|-h)r +
\sum_j r(I_j)$, from which the result follows.
\end{proof}

For any $l\geq 0$, set $[l]: = \{1,2,\ldots, l\}$ and let $\mathscr A_l$ be the set of multisets $\{a_1, \ldots, a_l\}$ of integers where $0\leq a_i \leq r-2$. If $A\in \mathscr A_l$,
we define $\Aut(A)$ to be the
group of permutations $\sigma:[l]\rightarrow [l]$ such that
$a_{i}=a_{\sigma(i)}$ for all $i\in [l]$.
For example, the multiset $A = \{1,2,2,3,3,3\}$ has
$|\Aut(A)| = 1!\cdot 2! \cdot 3! = 12$.

We are now ready to prove Theorem~\ref{thm:A} which we repeat here:

\begin{thm}\label{thm: open mirror symmetry for dimension one}
With the potential
\begin{equation}\label{flat one dim}
W_\ttts = \sum_{k \geq 0, l \geq 0} \sum_{\{a_i\} \in \mathscr{A}_l} (-1)^{l-1} \frac{\left\langle \prod_{i=1}^l \tau_0^{a_i} \sigma^{k+1}\right\rangle^{\tfrac1r, o}}{k!|\mathrm{Aut}(\{a_i\})|} \left(\prod_{i=1}^l t_{a_i} \right) x^{k}
\in \C[t_0,\ldots,t_{r-2},x],
\end{equation}
$\Omega=\textup{d}x$ is a primitive form and $t_0,\ldots,t_{r-2}$ are flat coordinates. Further, if $\mathcal{I}_d$ denotes the ideal in
$\C[t_0,\ldots,t_{r-2}]$ generated by all degree $d$ monomials, then
\begin{equation}
\label{eq:univ unfold}
W_\ttts = x^r+t_{r-2}x^{r-2}+\ldots+t_0 \mod \mathcal{I}_2.
\end{equation}
\end{thm}

\begin{proof}
Note that by Lemma \ref{twist Conditions}, all terms of $W_\ttts$
arising from invariants with internal markings are of degree
at most $r-1$ in the variable $x$. On the other hand, by
Theorem \ref{BCT Computation}, the only non-zero invariant
$\langle \sigma^{k+1}\rangle^{{1\over r},o}$ with no internal markings
has $k=r$ and the invariant is $-r!$. Thus $x^r$ appears in $W_\ttts$
and is the only monomial in $W_\ttts$ with degree greater than or equal to $ r$. That is,
$$
W_\ttts \equiv x^r \mod \mathcal{I}_1.
$$
A similar argument from Theorem \ref{BCT Computation} for invariants with
one internal marked point then gives \eqref{eq:univ unfold}.

We now expand
\begin{equation}\label{Exponentiating the perturbation}
\int_{\Xi_{d}}  e^{W_\ttts / \hbar} \Omega  = \int_{\Xi_{d}}  \left(\sum_{l\geq 0} \frac{ (W_\ttts-x^r)^l}{l!\hbar^l}\right) e^{x^r / \hbar} \Omega.
\end{equation}

Any monomial summand in the formal power series $\sum_{l \ge 1} \frac{(W-x^r)^l}{l! \hbar^l}$ will be of the form $c x^{nr+k} \hbar^{-l}$ for some $c \in \C [t_0, \ldots, t_{r-2}]$, $n \in \N$ and $k \in \{0,\ldots, r-1\}$. By Lemma \ref{lem:int by parts}, it thus follows that

$$
\int_{\Xi_d}  c x^{nr+k} \hbar^{-l}e^{x^r/\hbar}\Omega = \tilde c \hbar^{n-l} \int_{{\Xi_d}} x^k e^{x^r/\hbar}\Omega = \tilde c \hbar^{n-l} \delta_{dk}
$$
 for some $\tilde c \in \C [t_0, \ldots, t_{r-2}]$. Since $nr+k \le (r-1)l$, we have that $n < l$. Thus, the highest power of $\hbar$ appearing in $\int_{\Xi_{d}}  \left(\sum_{l\geq 1} \frac{ (W_\ttts-x^r)^l}{l!\hbar^l}\right) e^{x^r / \hbar} \Omega$ is $-1$.
Hence
\begin{equation}
\int_{\Xi_{d}}  e^{W_\ttts / \hbar} \Omega  = \delta_{0d} + \varphi_d \hbar^{-1} + \cdots
\end{equation}
where $\varphi_d \in \C[[t_0,\ldots, t_{r-2}]]$, for all $d$.
Note that the $\delta_{0d}$ term comes from~\eqref{eqn: dual basis} and
the $l=0$ term of the Taylor expansion in
\eqref{Exponentiating the perturbation}.

We will prove that
\begin{equation} \label{flat claim}
\varphi_d \equiv t_d \bmod \mathcal{I}_l
\end{equation}
for all $l$, and hence $\varphi_d=t_d$. First, by \eqref{eq:univ unfold},
it is clear that $\varphi_d = t_d \bmod \mathcal{I}_2$.

Consider a multiset $I = \{a_1, \ldots, a_l\}$. Suppose that there is
a multiset partition $I=\bigcup_{j=1}^h I_j$ of $I$ such that
there exists a non-negative integer $k_{I_j}$ with
$\langle \prod_{i \in I_j} \tau^{a_i} \sigma^{k_{I_j}+1}\rangle^{\tfrac1r, o}\neq 0$ for all $j$. Then we obtain in the expansion of $(W-x^r)^h$ a term of the form
\[
c x^{\sum_{j=1}^h k_{I_j}}\hbar^{-h} \prod_{j=1}^h\langle \prod_{i \in I_j} \tau^{a_i} \sigma^{k_{I_j}+1}\rangle^{\tfrac1r, o}
\]
where $c\in \C[t_0, \ldots, t_{r-2}]$. Note that
$k_{I_j}=r(I_j)$ by Lemma \ref{twist Conditions}, so the above
term, after taking the integral, will contribute a term
of the form $\hbar^{-1}$ if and only if $\sum_{j=1}^h r(I_j)=
r(I)+(h-1)r$. So let us assume this is the case, as we are not
interested in contributions of the form $\hbar^{-d}$, $d\ge 2$.
Then, again by Lemma \ref{twist Conditions} and this assumption,
\[
\sum_{i\in I} a_i=\sum_{j=1}^h \left[r(I_j)+(|I_j|-1)r\right]=
r(I)+(|I|-1)r
\]
and the hypotheses of Corollary \ref{Combo Corollary} hold.

On the way to proving that Equation~\eqref{flat claim} holds for all $l$, we will introduce the following notation. Let $\mathrm{Part}_h(I)$ denote the unordered partitions of $I$ into $h$ multisets. For example,
with $h=2$,
\[
\mathrm{Part}_h(\{1,1,2,2\})=\bigg\{\{\{1\},\{1,2,2\}\},
\{\{2\},\{1,1,2\}\}, \{\{1,1\},\{2,2\}\}, \{\{1,2\},\{1,2\}\}\bigg\}.
\]
On the other hand, we write $\mathrm{Part}_h([l])$ for partitions of
$[l]$ into $h$ disjoint sets. There is a map
\[
{\bf a}:\mathrm{Part}_h([l])\rightarrow \mathrm{Part}_h(I)
\]
given by ${\bf a}(\{Q_1,\ldots,Q_h\})=\{I_1,\ldots,I_h\}$, where
$I_j=\{a_i\,|\,i\in Q_j\}$. This map is surjective but not injective:
e.g., in the above example,
\[
{\bf a}(\{\{1,3\},\{2,4\}\})=
{\bf a}(\{\{1,4\},\{2,3\}\})=\{\{1,2\},\{1,2\}\}.
\]
For $\{I_1,\ldots,I_h\}\in \mathrm{Part}_h(I)$, we write $\Aut(
\{I_1,\ldots,I_h\})$ for the
set of permutations $\sigma:[h]\rightarrow [h]$ with $I_{\sigma(i)}=I_i$.

With this notation, note that
\begin{equation}
\label{eq:Aut}
|\Aut(I)|=|{\bf a}^{-1}(\{I_1,\ldots,I_h\})|\cdot |\Aut(\{I_1,\ldots,I_h\})|
\cdot \prod_{j=1}^h |\Aut(I_j)|.
\end{equation}

Using Equations~\eqref{flat one dim} and~\eqref{Exponentiating the perturbation}, we can see that the summand in the integral that corresponds to the coefficient of the monomial $t_I:= \prod_{a_i\in I} t_{a_i}$ is:
\[
t_I\sum_{h=1}^{|I|} \frac{1}{h!} \sum_{\{I_1, \ldots, I_h\} \in \mathrm{Part}_h(I)}\frac{h!}{\Aut(\{I_1, \ldots, I_h\})} x^{\sum_{j} r(I_j)} \hbar^{-h} \prod_{j=1}^h (-1)^{|I_j|-1} \frac{\langle \prod_{i \in I_j} \tau^{a_i} \sigma^{k_{I_j}+1}\rangle^{\tfrac1r, o}}{k_{I_j}! |\Aut(I_j)|},
\]
where $k_{I_j}=r(I_j)$ as above.

Using Corollary~\ref{Combo Corollary}(b) combined with
Lemma \ref{lem:int by parts},
we can see that the only integral that will have a nonzero contribution is that corresponding to the cycle $\Xi_{r(I)}$ and that the $\hbar^{-1}$ term will have a summand of the form $\Lambda_I t_I$ where
\begin{equation}
\Lambda_I : = \sum_{h=1}^{|I|} \frac{1}{h!} \sum_{\{I_1, \ldots, I_h\} \in {\mathrm{Part}}_h(I)} \frac{h!}{\Aut(\{I_1,\ldots,I_h\})}(-1)^{h-1} \frac{\Gamma(\tfrac{1+ \sum_j k_{I_j}}{r})}{\Gamma(\frac{1+k_I}{r})} \prod_{j=1}^h (-1)^{|I_j|-1} \frac{\langle \prod_{i \in I_j} \tau^{a_i} \sigma^{k_{I_j}+1}\rangle^{\tfrac1r, o}}{k_{I_j}! |\Aut(I_j)|}.
\end{equation}
We remind the reader that $k_I=r(I)$ by Lemma \ref{twist Conditions}.
By using the closed form for the open $r$-spin invariants given Theorem~\ref{BCT Computation} and cancelling signs, we then have that
\begin{equation}
\Lambda_I = \sum_{h=1}^{|I|} \frac{1}{h!} \sum_{\{I_1, \ldots, I_h\} \in {\mathrm{Part}}_h(I)} \frac{h!}{\Aut(\{I_1,\ldots,I_h\})}(-1)^{h-1} \frac{\Gamma(\tfrac{1+ \sum_j k_{I_j}}{r})}{\Gamma(\frac{1+k_I}{r})} \prod_{j=1}^h  \frac{(k_{I_j} + |I_j| -1)! }{k_{I_j}! |\Aut(I_j)| r^{|I_j|-1} }
\end{equation}
By \eqref{eq:Aut},
\begin{equation}
\label{eq:more auts}
{1\over |\Aut(\{I_1,\ldots,I_h\})|\cdot\prod_{j=1}^{|I|} |\Aut(I_j)|}
= {|{\bf a}^{-1}(\{I_1,\ldots,I_h\})|\over |\Aut(I)|}
\end{equation}
and hence by replacing the sum over elements of
$\mathrm{Part}_h(I)$ with the larger sum over elements
of $\mathrm{Part}_h([l])$ and multiplying by $|\Aut(I)|$, we obtain the following simplification:
\begin{equation}
|\Aut(I)|\Lambda_I =\sum_{h=1}^{|I|}  \sum_{\{Q_1, \ldots, Q_h\}
\in \mathrm{Part}_h([l])} (-1)^{h-1} \frac{\Gamma(\tfrac{1+ \sum_j k_{I_j}}{r})}{\Gamma(\frac{1+k_I}{r})} \prod_{j=1}^h   \frac{(k_{I_j} + |I_j| -1)! }{k_{I_j}! r^{|I_j|-1} },
\end{equation}

We claim that $|\Aut(I)|\Lambda_I $ vanishes. To prove this, we will instead first view the twists $a_i$ as formal variables and then change coordinates to new formal variables $b_i:= r-a_i$. We set the notation $b_A := \sum_{i \in A} b_i$. Note that, since $\sum_{i\in A} a_i = r(A) + (|A|-1) r$ for all subsets $A \subset I$,
$b_A=r-r(A)$ and
the number $b_A$ is in the set $\{0,\ldots, r\}$.

Note now that we can rewrite $|\Aut(I)|\Lambda_I$ in the following way:
\begin{equation}
|\Aut(I)|\Lambda_I=  \sum_{h=1}^{|I|}  \sum_{\{Q_1, \ldots, Q_h\} \in \mathrm{Part}_h([l])} (-1)^{h-1} \frac{\Gamma(h+ \tfrac{1-  b_I}{r})}{\Gamma(\frac{1+r-b_I}{r})} \prod_{j=1}^h  \frac{ (r-1+ |I_j| - b_{I_j})! }{(r-b_{I_j})! r^{|I_j|-1} }.
\end{equation}
This is further rearranged as:
\begin{equation}\begin{aligned}\label{last lambda I}
|\Aut(I)|\Lambda_I &= \sum_{h=1}^{|I|} \frac{(-1)^{h-1}}{r^{l-1}} \sum_{\{Q_1, \ldots, Q_h\} \in \mathrm{Part}_h([l])} \left(\prod_{j=1}^{h-1}(jr+1-b_I)\right)\left(\prod_{j=1}^h  \frac{ (r-1+ |I_j| - b_{I_j})! }{(r-b_{I_j})!  }\right)\\
	&= \frac{1}{r^{l-1}} \sum_{h=1}^{|I|}(-1)^{h-1} \sum_{\{Q_1, \ldots, Q_h\} \in \mathrm{Part}_h([l])} \left(\prod_{j=1}^{h-1}(jr+1-b_I)\right)\left(\prod_{j=1}^h \prod_{i=1}^{|I_j|-1}(r-b_{I_j} + i)\right).
\end{aligned}\end{equation}

From~\eqref{last lambda I}, it is transparent that $|\Aut(I)|\Lambda_I$ is a polynomial in the variables $b_i$. If we show that $|\Aut(I)|\Lambda_I$ is the zero polynomial in the variables $b_i$, then this implies that $\Lambda_I$ vanishes. Therefore, the theorem reduces to showing that $|\Aut(I)|\Lambda_I$ is the zero polynomial for any $I$ with cardinality at least two. We proceed by induction.

Let $l = 2$ and consider the multiset $I = \{a_1, a_2\}$. Then $\mathrm{Part}_1([l]) = \{I\}$ and $\mathrm{Part}_2([l])$ only contains the one (unordered) partition $\{\{a_1\},\{a_2\}\}$. In this case, \eqref{last lambda I} reduces to:
$$
|\Aut(I)|\Lambda_I = \frac{1}{r} (r-b_I +1) - \frac{1}{r}(r+1 - b_I) = 0.
$$

Moving to the case where $I$ has cardinality greater than two, by the induction hypothesis, we know for any subset $I' \subset I$ of cardinality above $1$ that $\Lambda_{I'} = 0$ viewed as a polynomial in the variables $b_i$.

Note that $|\Aut(I)|\Lambda_I$ is a symmetric polynomial of degree at most $l-1$ in $l$ variables. Recall that for an arbitrary polynomial $P_l$ of degree at most $l-1$ in $l$ variables $b_i$ that if $P_l|_{\{b_i=0\}} = 0$ for all $i$ then the polynomial $P_l$ is the zero polynomial. Since $\Lambda_I$ is symmetric, it suffices to show that $\Lambda_I|_{b_l=0}=0$. Note that $\Lambda_I|_{b_l=0}$ is a polynomial in $l-1$ variables, and we will now work to relate it to the polynomial $\Lambda_{I \setminus\{a_l\}}$.

Take $\mathrm{Part}([l]) = \bigcup_{h=1}^l \mathrm{Part}_h([l])$ and let
$\mathrm{Part}([l-1]) = \bigcup_{h=1}^{l-1} \mathrm{Part}_{h}([l-1])$.
We then can define a surjective function $f : \mathrm{Part}([l]) \rightarrow
\mathrm{Part}([l-1])$ by:
\begin{equation}
f(\{Q_1,\ldots, Q_h\}) = \begin{cases} \{Q_1, \ldots, Q_j\setminus \{l\}, \ldots, Q_h\} & \text{if $l \in Q_j$ and $|Q_j| >1$}; \\ \{Q_1, \ldots, \widehat{Q_j}, \ldots, Q_h\} & \text{if $Q_j = \{l\}$}.  \end{cases}
\end{equation}
We will denote by $\mathrm{Cont}(\{Q_1, \ldots, Q_h\})$ the polynomial
\[
\mathrm{Cont}(\{Q_1, \ldots, Q_h\}) := (-1)^{h-1}\left(\prod_{j=1}^{h-1}(jr+1-b_I)\right)\left(\prod_{j=1}^h \prod_{i=1}^{|I_j|-1} r-b_{I_j} + i\right),
\]
which is the contribution to $\Lambda_I$ given by the partition $\{Q_1, \ldots, Q_h\}$, where
\[
\{I_1,\ldots,I_h\}={\bf a}(\{Q_1,\ldots,Q_h\}).
\]
We compare the contribution $\mathrm{Cont}(Q)$ for $Q \in \mathrm{Part}([l-1])$ with the sum $\sum_{Q' \in f^{-1}(Q)} \mathrm{Cont}(Q')$. We claim:
\[
\sum_{Q' \in f^{-1}(Q)} \mathrm{Cont}(Q')|_{b_l=0} = (l-2) \mathrm{Cont}(Q).
\]
Indeed, take $Q = \{Q_1, \ldots, Q_h\} \in \mathrm{Part}_h([l-1])$,
with ${\bf a}(Q)=\{I_1,\ldots,I_h\}$. Its inverse image under $f$ consists of exactly $h+1$ partitions, namely  $Q^0 = \{Q_1,\ldots, Q_h, \{a_l\}\}$ and $Q^j = \{Q_1, \ldots, Q_j\cup\{l\}, \ldots, Q_h\}$ for $1\leq j \leq h$. Note that
\begin{equation}\begin{aligned}
\mathrm{Cont}(Q^0)|_{b_l=0} &= -(hr+1-b_{I\setminus \{a_l\}}) \mathrm{Cont}(Q)\\
\mathrm{Cont}(Q^j)|_{b_l=0} &= (r - b_{I_j} + |Q_j|) \mathrm{Cont}(Q) \text{ for $1\leq j \leq h$}.
\end{aligned}\end{equation}
Hence, by using that $\sum_{j=1}^h |Q_j| = l-1$, we see that
$$
\sum_{Q' \in f^{-1}(Q)} \mathrm{Cont}(Q')|_{b_l=0} = \sum_{j=0}^h \mathrm{Cont}(Q^j)|_{b_l=0} = (l-2) \mathrm{Cont}(Q).
$$
We now can use this recursion with Equation~\eqref{last lambda I}, the surjectivity of $f$, and the induction hypothesis to see that
\begin{equation}\begin{aligned}
(|\Aut(I)|\Lambda_I)|_{b_l=0} &=  \frac{1}{r^{l-1}}  \sum_{Q \in \mathrm{Part}([l])}
\mathrm{Cont}(Q)|_{b_l=0} \\
	&= \frac{1}{r^{l-1}} \sum_{Q \in \mathrm{Part}([l-1])} (l-2) \mathrm{Cont}(Q)\\
	&= \frac{1}{r^{l-1}} \left(r^{l-2}(l-2) |\Aut(I\setminus \{a_l\})|\Lambda_{I\setminus\{a_l\}}\right)\\
	&= 0
\end{aligned}\end{equation}
This consequently implies that $\Lambda_I=0$ for all $|I| = l\ge 2$. Therefore,
$\Omega$ is a primitive form and the
coordinates $t_0,\ldots,t_{r-2}$ are flat coordinates.
\end{proof}

\bibliography{biblio_fermat}
\bibliographystyle{amsalpha}
\end{document}